\documentclass[12pt, reqno]{amsart}

\usepackage{amssymb,amsmath,amsthm,mathrsfs}
\usepackage{enumitem}
\usepackage{pgf}
\usepackage{tikz}
\usetikzlibrary{arrows,automata}
\usepackage[latin1]{inputenc}

\renewcommand{\Re}{\textup{Re }}
\newcommand{\ignore}[1]{}

\renewcommand{\mod}[1]{
{\ifmmode\text{\rm\ (mod~$#1$)}
\else\discretionary{}{}{\hbox{\!\!}}\rm(mod~$#1$)\fi}}

\newcommand{\mock}[1]{
{\ifmmode\text{\rm\ (mock~$#1$)}
\else\discretionary{}{}{\hbox{\!\!}}\rm(mock~$#1$)\fi}}

\newcommand{\Z}{\mathbb Z}

\newcommand{\N}{\mathbb N}

\newcommand{\C}{\mathbb C}
\newcommand{\F}{\mathbb F}
\newcommand{\U}{\mathbb U}
\newcommand{\D}{\mathbb D}

\newcommand{\leg}[2]{{#1 \overwithdelims() #2}}

\newsymbol\dnd 232D

\advance \topmargin by -\headheight
\advance \topmargin by -\headsep
\evensidemargin \oddsidemargin
\title{Mock characters and the Kronecker symbol}

\author{Jean-Paul Allouche}
\address{CNRS, Institut de Math\'ematiques de Jussieu-PRG, \'Equipe Combinatoire et Optimisation, Universit\'e Pierre et Marie Curie, Case 247, 4 Place Jussieu, F-75252 Paris Cedex 05 France}
\email{jean-paul.allouche@imj-prg.fr}

\author{Leo Goldmakher}
\address{Dept of Mathematics \& Statistics,
Williams College, Williamstown, MA, USA}
\email{leo.goldmakher@williams.edu}

\thanks{JPA partially supported by the ANR project ``FAN'' (Fractals et Num\'eration), ANR-12-IS01-0002. LG partially supported by an NSA Young Investigator grant.}

\theoremstyle{plain}
\newtheorem{theorem}{Theorem}[section]
\newtheorem{lemma}[theorem]{Lemma}
\newtheorem{corollary}[theorem]{Corollary}
\newtheorem{proposition}[theorem]{Proposition}
\newtheorem{conj}[theorem]{Conjecture}

\theoremstyle{definition}

\newtheorem{remark}[theorem]{Remark}
\newtheorem{example}[theorem]{Example}
\newtheorem{definition}[theorem]{Definition}

\numberwithin{equation}{section}

\begin{document}

\maketitle

\begin{abstract}
We introduce and study a family of functions we call the \emph{mock characters}. These functions satisfy a number of interesting properties, and of all completely multiplicative arithmetic functions seem to come as close as possible to being Dirichlet characters. Along the way we prove a few new results concerning the behavior of the Kronecker symbol.
\end{abstract}

\medskip


One of the most familiar objects in number theory is the Kronecker symbol, denoted $\leg{a}{n}$ or $(a | n)$.
Viewed as a function of $n$, it is well-known that this is a primitive real character when $a$ is a fundamental
discriminant. Less well-known is the behavior when $a$ is \emph{not} a fundamental discriminant; in this case
$\leg{a}{\cdot}$ might be a primitive character (e.g., for $a = 2$), an imprimitive character (e.g., for $a = 4$),
or not a character at all (e.g., for $a = 3$).\footnote{See Section~\ref{Sect:KronSymbol} below, in particular,
Corollary~\ref{KroneckerEqualsDirichlet}.}
Even when $\leg{a}{\cdot}$ is not a character, it strongly mimics the behavior of one, replacing the condition
of periodicity with \emph{automaticity} (a notion we shall discuss below). Inspired by this example, we define
and study a family of character-like functions which we call \emph{mock characters}. Of all completely multiplicative
arithmetic functions, the mock characters are as close as possible to being characters. We will justify this statement
qualitatively, and also formulate a quantitative conjecture using the language of `pretentiousness' introduced by
Granville and Soundararajan in \cite{GranSound}.

The structure of the paper is as follows.
In Section~\ref{glance} we briefly review automatic sequences.
In Section~\ref{sect:MockChar} we introduce mock characters and prove a few basic properties.
In Section~\ref{sect:Kronecker} we explore the relationship between mock characters and the Kronecker symbol, and
use our results to obtain some new results about the Kronecker symbol.
In the final section, we view mock characters through the lens of the `pretentious' approach to number theory.

\section{A quick overview of automatic sequences}\label{glance}

The goal of this section is to recall and motivate the notion of an automatic sequence. We will first give a
computer-science definition, then discuss a mathematical motivation for studying automatic sequences, and finally
give an equivalent definition which is easier to compute with.

Recall that a \emph{Finite State Machine} is a finite collection of states along with transition rules
between states. A positive integer corresponds to a program: its digits (read right to left) dictate the
individual state transitions.

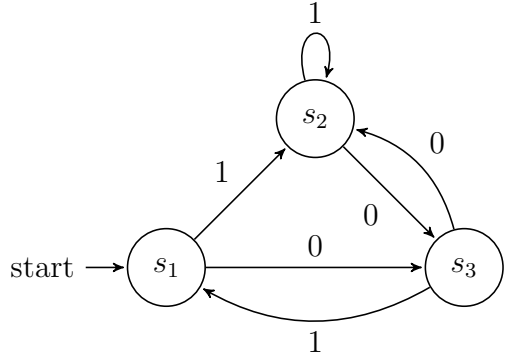
\begin{figure}[h] 
\centering
\begin{tikzpicture}[->,>=stealth',shorten >=1pt,auto,node distance=2.8cm,semithick]
  \node[initial,state] (A)                    {$s_1$};
  \node[state]         (B) [above right of=A] {$s_2$};
  \node[state]         (C) [below right of=B] {$s_3$};

  \path (A) edge              node {1} (B)
            edge [above]      node {0} (C)
        (B) edge [loop above] node {1} (B)
            edge [below left] node {0} (C)
        (C) edge [bend left, below] node {1} (A)
            edge [bend right, above right]  node {0} (B);
\end{tikzpicture}
\caption{For this FSM, the binary program 1101 yields the output $s_2$} \label{fig:FSM}
\end{figure}

\begin{definition}\label{defn:CSDefnAuto}
A sequence $(a_n)$ is called \emph{$q$-automatic} if and only if there exists a Finite State Machine
which outputs $a_n$ when given the program $n$ (written base $q$).
\end{definition}

\begin{example}\label{ex:FirstDefPaper}
A nice example of a 2-automatic sequence is the \emph{regular paperfolding sequence}. To generate this sequence,
start with a large piece of paper, and label the top left corner $L$ and the top right corner $R_0$.
Fold the paper in half so that $R_0$ ends up underneath $L$, and label the new upper right corner by $R_1$
(the paper should be oriented so that $L$ is still the top left corner). Now iterate this process, each time
folding the paper in half so that the top right corner $R_i$ ends up directly underneath $L$, and then labelling
the new top right corner $R_{i+1}$. After some number of iterations, unfold the paper completely so that $L$ is
the top left corner and $R_0$ is once again the top right corner.
The paper has a sequence of creases, some of which are mountains (denoted $\wedge$) and some valleys (denoted $\vee$).
Set $v_n$ to be the $n^\text{th}$ crease from the left, where the paper has been folded enough times for there to be
$n$ creases.
The sequence $(v_n)$ is called the (regular) paperfolding sequence; it begins
\[
\wedge
\wedge
\vee
\wedge
\wedge
\vee
\vee
\wedge
\wedge
\wedge
\vee
\vee
\wedge
\vee
\vee
\cdots
\]
For more information about this sequence see \cite{AS}. (An FSM generating the paperfolding sequence can
be found in \cite[p.\ 312]{AMF}; note that the first term of the sequence is treated as the $0^\text{th}$
term.) We will return to this sequence below.
\end{example}

 From the definition it is not obvious how restrictive the condition of automaticity is; it turns out to be
quite strong. To justify this, first recall that a given $\alpha \in \mathbb{R}$, say with decimal expansion
$\displaystyle \alpha = \sum_{n \geq 0} \alpha_n 10^{-n}$, is rational if and only if the sequence of digits
$(\alpha_n)$ is eventually periodic. Similarly, given a finite field $\mathbb{F}$ and a formal power series
$\alpha(X) \in \mathbb{F}[[X]]$, say, $\displaystyle {\alpha(X) = \sum_{n \geq 0} \alpha_n X^n}$,
we have that $\alpha(X) \in \mathbb{F}(X)$ (i.e., $\alpha(X)$ is rational) if and only if the sequence
$(\alpha_n)$ is eventually periodic.

It is an interesting and largely open question to determine whether irrational algebraic numbers have
``random'' decimal expansions or not. In the function field case, when the ground field is finite,
Christol discovered the following result:

\begin{theorem}[\cite{C1, C2}]\label{christol}
Let $\F_q$ denote the finite field of $q$ elements. Then the formal power series
$\displaystyle \alpha(X) := \sum_{n \geq 0} \alpha_n X^n$ is algebraic over the
field of rational functions ${\mathbb F}_q(X)$ if and only if the sequence $(\alpha_n)$
is $q$-automatic.
\end{theorem}

\noindent
{\samepage Thus we see that
\begin{equation}\label{analogy}
\textit{automatic is to periodic}
\qquad \textit{as} \qquad
\textit{algebraic is to rational.}
\end{equation}
This analogy will play a role in the sequel.}

Definition~\ref{defn:CSDefnAuto} is clean and justifies the term \emph{automatic}. In practice, however,
the following equivalent definition (see, e.g., \cite{AS}) is more useful:

\begin{definition}
Let $q \geq 2$ be an integer. The sequence $\big(\alpha(n)\big)_{n \geq 0}$ is said to be
\emph{$q$-automatic} if its {\em $q$-kernel}, the collection of subsequences
\[
{\mathcal K}_{\alpha} := \Big\{\big(\alpha(q^k m + r)\big)_{m \geq 0} : \ k \geq 0 \text{ and } 0 \leq r < q^k\Big\} ,
\]
is finite. We emphasize that an individual element of ${\mathcal K}_{\alpha}$ is a \emph{sequence}.
\end{definition}

\begin{remark}
It follows from this that any eventually periodic sequence is $q$-automatic for every $q \geq 2$.
\end{remark}

\section{Mock characters}\label{sect:MockChar}

\noindent
Before defining the notion of mock character, we recall the definition of a Dirichlet character.

\begin{definition}\label{char}
A map $\chi : \Z \to \C$ is a \emph{Dirichlet character of modulus $q$}, denoted $\chi\mod{q}$,
if and only if

\begin{enumerate}[label=(\roman*)]

\item $\chi$ is completely multiplicative (i.e., $\chi(mn) = \chi(m)\chi(n)$ for all $m,n \in \Z$);

\item $\chi$ is $q$-periodic; and

\item $\chi(n) = 0$ if and only if $(n,q) \neq 1$.

\end{enumerate}
\end{definition}

\noindent
In fact, the third condition is merely a notational convenience; any completely multiplicative periodic function is a Dirichlet character. This does not seem to be well-known and the proof is short, so we include it here. (This is a modification of a theorem appearing in the second author's Ph.D. thesis \cite{GoldThesis}.)

\begin{proposition}\label{prop:CharFirstTwoProp}
Suppose $\chi : \Z \to \C$ is completely multiplicative and eventually periodic, and that there exists $n > 1$ such that $\chi(n) \neq 0$. Then $\chi$ is a Dirichlet character.
\end{proposition}

\begin{proof}
Observe that $\chi(1) = 1$. If $\chi$ is the trivial character of modulus 1 the proposition is proved, so we will henceforth assume that $\chi\not\equiv 1$; it follows that $\chi(0) = 0$.

We begin by proving a special case of the proposition: that the claim holds for all purely periodic $\chi$. Let $q$ denote the period of $\chi$, and set $d$ to be the largest divisor of $q$ such that $\chi(d) \neq 0$ (such a $d$ exists since $\chi(1) = 1$).
Observe that if $d = 1$ then $\chi$ must be a Dirichlet character \mod{q}. (To see this, note that $\chi(m) = 0$
whenever $(m,q) > 1$; conversely, if $(m,q) = 1$, then some linear combination of $m$ and $q$ yields 1,
whence $\chi(m)$ must be nonzero.) Thus we may assume we have $d > 1$. For any integers $k,r$ we have
\[
\chi(d) \chi\Big(\frac{kq}{d} + r\Big) =
\chi(kq + dr) =
\chi(d)\chi(r) .
\]
Now $\chi(d) \neq 0$ by hypothesis, so $\chi$ is periodic with period $\frac{q}{d}$.
Since $d > 1$, we see that $\frac{q}{d}$ is a proper divisor of $q$. Iterating this procedure yields the claim in the case that $\chi$ is periodic.

We now deduce the claim in full generality. Let $\chi$ be as in the statement of the proposition. A result of Heppner and Maxsein \cite[Satz 2]{HeppnerMaxsein} implies the existence of an integer $\ell \geq 0$ and a purely periodic multiplicative function $\xi$ such that
$
\chi(n) = n^\ell \xi(n) .
$
Since $\chi$ is completely multiplicative by hypothesis, we deduce that $\xi$ must be as well. Thus $\xi$ is periodic and completely multiplicative, so it must be a Dirichlet character.
Finally, observe that $\chi(n)$ takes only finitely many values, so complete multiplicativity implies that all these values must be 0 or roots of unity. Thus (by hypothesis) there exists $n > 1$ such that $|\chi(n)| = 1$, whence $\ell = 0$. It follows that $\chi = \xi$ is a Dirichlet character.
\end{proof}

\begin{remark}
Heppner and Maxsein's proof is an adaptation of a proof of a similar result (for purely periodic functions) due to S\'{a}rk\"{o}zy \cite{Sarkozy}. A much shorter proof of the Heppner-Maxsein theorem was discovered by Methfessel \cite[Theorem 3]{Methfessel}.
Although not required for the sequel, we mention that the S\'{a}rk\"{o}zy-Heppner-Maxsein-Methfessel theorem implies significantly more than we have indicated. For example, without any additional effort one can deduce that a completely multiplicative function $\chi$ which satisfies $\chi(n) \neq 0$ for some $n > 1$ and which eventually satisfies \emph{any} linear recurrence must be a Dirichlet character.
\end{remark}

The simultaneous requirements of complete multiplicativity and (eventual) periodicity are very strong. Slightly weakening the latter condition allows us to enlarge the family of Dirichlet characters.
To distinguish the new members of this family, the definition below excludes all Dirichlet characters.

\begin{definition}\label{mock}
A map $\kappa : \Z \to \C$ is a \emph{mock character of mockulus $q$}, denoted $\kappa \mock{q}$,
if and only if

\begin{enumerate}[label=(\roman*)]

\item $\kappa$ is completely multiplicative;

\item the sequence $(\kappa(n))_{n \geq 0}$ is $q$-automatic but not eventually periodic; and

\item there exists an integer $d \geq 1$ such that $\kappa(n) = 0$ precisely when $n=0$ or $(n,d) \neq 1$.

\end{enumerate}
\end{definition}

\begin{example}\label{ex:PaperFold}
We saw the paperfolding sequence $(v_n)_{n \geq 1}$ in Example~\ref{ex:FirstDefPaper}. We now reinterpret and
extend this sequence to all integers. First, replace every occurrence of $\wedge$ by $+1$ and every occurrence of
$\vee$ by $-1$.
It turns out that $(v_n)$ satisfies a nice recursion: $v_{2n} = v_n$, and $v_{2n+1} = (-1)^n$
(see, e.g., \cite[Section~6]{AS}).
Next, extend the sequence to all of $\Z$ by setting $v_0 :=0$ and $v_{-n} := -v_n$.
One can show that the sequence is completely multiplicative.
It is not hard to deduce that the function $v(n) := v_n$ is a mock character of mockulus $2$.
\end{example}

\begin{remark}
Our analogy \eqref{analogy} gives a qualitative heuristic that of all completely multiplicative arithmetic functions, the mock characters come as close as possible to being Dirichlet characters. In Section \ref{sect:Kronecker} we will further justify this by example; in Section \ref{Sect:FinalSection} we
propose a conjecture which quantifies this heuristic description.
\end{remark}

\begin{remark}
Other notions of ``character-like'' functions and ``generalized characters'' appear in the literature, which are
quite different from the mock characters defined above; see \cite{Bronstein, Cudakov, CL, CR, Vandehey}.
On the other hand, multiplicative automatic sequences (not exactly mock characters, but close) have been studied in a number of recent papers, e.g.,
\cite{BBC, BCH, BC, BCC, Coons, SP2, Yazdani}.
One notable conjecture, made in \cite{BBC}, is that any function which is simultaneously automatic and multiplicative
must agree at all primes with some eventually periodic function.
\end{remark}

We make a few observations about mock characters.
First, note that condition~(iii) in Definition~\ref{mock} implies that the set of primes $p$ for which $\kappa(p)=0$
is finite. Further, note that $\kappa$ is nonvanishing on $\Z^+$ (the positive integers) if and only if $d=1$ in
condition~(iii).
Less obvious are the following:

\begin{lemma} \label{remark-mock}
Suppose $\kappa$ is a mock character of mockulus at least $2$.
\begin{enumerate}
\item $\kappa$ has mockulus $q$ if and only if it has mockulus $q^r$ for any positive integer $r$.

\item For all $n \in \Z$, $\kappa(n)$ is either zero or a root of unity.
\end{enumerate}
\end{lemma}
\begin{proof}
The first assertion is a classical result about automatic sequences; see, e.g., \cite{AS}.
To prove the second assertion, note that $\kappa$ only takes finitely many values (since it is automatic).
By complete multiplicativity, it follows that all these values must be $0$ or roots of unity.
\end{proof}

Definition~\ref{mock} makes it seem that a mock character depends on two different parameters:
the mockulus $q$, and the integer $d$ appearing in condition (iii). We now give an equivalent
definition of mock character which eschews this complication by showing that the only required
parameter is the mockulus. In practice, however, Definition~\ref{mock} is the simpler one to verify.

\begin{proposition}
Let $q$ be an integer $\geq 2$. A map $\kappa : \Z \to \C$ is a mock
character \mock{q}
if and only if

\begin{enumerate}[label=(\Roman*)]

\item $\kappa$ is completely multiplicative;

\item the sequence $(\kappa(n))_{n \geq 0}$ is $q$-automatic but not eventually periodic; and

\item the series
$\displaystyle\sum_{\substack{p \ {\rm prime} \\ \kappa(p)=0}} \frac{1}{p}$ converges.

\end{enumerate}
\end{proposition}

\begin{proof}
The fact that ((i), (ii), (iii)) imply ((I), (II), (III)) is trivial. We will prove the
reverse implication by showing that (I), (II), and (III) imply that the set of $p$ for which
$\kappa(p) = 0$ is finite. Taking $d$ to be the product of these primes, we will deduce (iii).

Consider the map $|\kappa|$. This map only takes the values $0$ and $1$ since the nonzero values
of $\kappa$ are roots of unity (see Lemma~\ref{remark-mock} above).
Let $a_n$ denote the $n$th smallest element of the
set $\{a \in \N : \kappa(a) \neq 0\}$.
Applying a theorem of Delange \cite[Th\'eor\`eme~2]{Delange} with $f = |\kappa|$, we see that (III)
and (I) imply that $|\kappa|$ admits a nonzero mean value, say $M$. This immediately implies that
$a_n \sim n/M$ as $n \to \infty$, whence $a_{n+1}/a_n \to 1$ as $n \to \infty$.
Now, using Corollary~3 of \cite{SP1} and noting that our $|\kappa|$ is completely multiplicative
(so that the condition $q \mid f(p^{h_p})$ in that paper boils down to $\kappa(p) = 0$),
we see that there exist only finitely many primes $p$ for which $\kappa(p) = 0$.
\end{proof}

We conclude this section with two theorems about mock characters, which are inspired by (and improve upon)
some nice results in \cite{BCH,SP2} on nonvanishing completely multiplicative automatic functions.

\begin{theorem}\label{simple}
Let $q \geq 2$ be an integer and $f$ be a mock character \mock{q}.
Suppose that $f(p) \neq 0$ for some prime $p$ dividing $q$.
Then $q = p^m$ for some integer $m \geq 1$, and $f$ is a mock character \mock{p}.
\end{theorem}

\begin{proof}
The proof will follow the proof of \cite[Proposition~3.3]{BCH}, where the unnecessary hypothesis
that $f$ never vanishes is used. We restrict our function to the sequence $\big(f(n)\big)_{n \geq 0}$.
Recall that $p \mid q$, and set $q_1 := q/p$.
Then for any $k \geq 0$ and $r \in [0, q_1^k-1]$ we have
\[
f(p)^k f(q_1^k n + r)
= f(q^k n + p^k r).
\]
Now, since $p^k r$ belongs to $[0, q^k-1]$, the sequence $\big(f(q^k n + p^k r)\big)_{n \geq 0}$
belongs to ${\mathcal K}_f$ (the $q$-kernel of $f$), and hence belongs to a finite set of sequences.
Since $f(p) \neq 0$, $f(p)$ must be a root of unity (see Lemma~\ref{remark-mock});
thus $f(p)^k$ takes only finitely many values that are all nonzero.
Finally the $q_1$-kernel of $f$,
\[
{\mathcal K}_{q_1} =
\big\{\big(f(q_1^k n + r)\big)_{n \geq 0} : \ k \geq 0, \ r \in [0, q_1^k-1]\big\} ,
\]
is finite, which means that the sequence $f$ is $q_1$-automatic. But it is also $q$-automatic!
According to a deep theorem of Cobham \cite{Cobham} this implies, since $f$ is not periodic
(it is a mock character), that $q$ and $q_1$ cannot be multiplicatively independent.
Thus $q$ and $q_1$ are both powers of the same number; we conclude that $q = p^m$ for some
integer $m \geq 1$. By Lemma~\ref{remark-mock}, $f$ is a mock character of mockulus $p$.
\end{proof}

\begin{theorem} Let $q$ be an integer $\geq 2$, and suppose the mock character $f\mock{q}$ does not
vanish on $\Z^+$ (the positive integers). Then there exists a root of unity $\xi$, a prime $p$, a
positive integer $r$, and a Dirichlet character $\chi \mod{p^r}$ such that for all $n \geq 1$,
\[
f(n) = \xi^{v_p(n)} \chi\bigg(\frac{n}{p^{v_p(n)}}\bigg) .
\]
(Here $v_p(n)$ denotes the largest integer $t$ such that $p^t \mid n$.)
Moreover, $f$ has mockulus $p$.
\end{theorem}

\begin{proof}
By Theorem~\ref{simple}, we may assume that $q=p$ is prime. Using \cite[Proposition~1]{SP2}
we know that there exists an integer $k$,
such that if $n_1, n_2, \ell$ are integers with $(n_1, p^{\ell + 1}) \mid p^{\ell}$ and
$n_1 \equiv n_2 \mod{p^{k+\ell}}$, then $f(n_1) = f(n_2)$. But $p$ is prime, so that the
condition $(n_1, p^{\ell + 1}) \mid p^{\ell}$ boils down to $v_p(n_1) \leq \ell$. Taking
$\ell = 0$, we see (as also noted in \cite{BCH}) that there exists a Dirichlet character
$\chi \mod{p^k}$ such that if $(n,p) = 1$ then $f(n) = \chi(n)$. Now, for any $n \geq 1$,
we have $f(n) = f(p^{v_p(n)})f(n/p^{v_p(n)}) = f(p)^{v_p(n)} \chi(n/p^{v_p(n)})$. Letting
$\xi$ denote the value of $f(p)$, we conclude.
\end{proof}

\section{Kronecker symbols are (mock) characters} \label{sect:Kronecker}

For the remainder of the text, we set
\[
\kappa_a(n) :=
{\leg{a}{n}}
\]
where the right hand side is the Kronecker symbol.
After briefly recalling the definition and basic properties of the Kronecker symbol, we show that $\kappa_a$
is either a Dirichlet or a mock character. We apply the machinery we develop to prove some results about
generating functions of Kronecker symbols.

\subsection{The Kronecker symbol}

We briefly recall the definition of the Kronecker symbol $\leg{a}{n}$.
For convenience, we shall use Conway's convention \cite{Conway} that $-1$ is a prime.
First, set
\[
\phantom{\qquad \text{whenever } an = 0  \text{ or } (|a|,|n|) > 1.}
\leg{a}{n} = 0
\qquad \text{whenever } an = 0 \text{ or } (|a|,|n|) > 1.
\]
It therefore remains to define $\leg{a}{n}$ for nonzero coprime integers $a$ and $n$. Write $n$ in the form
\[
n = \prod_p p^{\nu_p}
\]
where $\nu_p \in \N$ (recall that $-1$ is considered prime).
We then set
\begin{equation}
\label{eq:DefnKronSymb}
\leg{a}{n} := \prod_p \leg{a}{p}^{\nu_p}
\end{equation}
where the $\leg{a}{p}$'s are defined as follows. (Keep in mind that we are assuming $(|a|,|p|) = 1$.)
For every $p \geq 3$, set
\[
\leg{a}{p} :=
\begin{cases}
1 &\quad \mbox{\rm if $a$ is a square modulo $p$, and} \\
-1 &\quad \mbox{\rm otherwise.}
\end{cases}
\]
For the remaining two primes, set
\[
\leg{a}{2} := \leg{2}{a}
\qquad \text{and} \qquad
\leg{a}{-1} :=
\begin{cases}
1 &\quad \mbox{\rm if $a > 0$} \\
-1 &\quad \mbox{\rm if $a < 0$}
\end{cases}
\]
Note that $\leg{a}{2}$ as defined above must be evaluated recursively using \eqref{eq:DefnKronSymb}.
However, there is also an explicit formula: for any odd $a$,
\[
\leg{a}{2} = (-1)^{\frac{a^2-1}{8}}
\]
In the special case that $n$ is odd, $\leg{a}{n}$ is called the \emph{Jacobi symbol}; when $n$ is a
positive odd prime, $\leg{a}{n}$ is called the \emph{Legendre symbol}.

A fundamental property of the Kronecker symbol (which we shall require in the sequel) is
\emph{quadratic reciprocity}: for any nonzero integers $m$ and $n$,
\begin{equation}\label{eq:QuadRecip}
\leg{m}{n}
= \sigma(m,n) \cdot
(-1)^{\frac{m_1 - 1}{2} \cdot \frac{n_1 - 1}{2}}
\leg{n}{m}
\end{equation}
where
$m_1$ and $n_1$ are the largest odd factors of $m$ and $n$, respectively
(i.e., $m_1 = m/2^{v_2(m)}$, $n_1 = n/2^{v_2(n)}$) and
\[
\sigma(m,n) =
\begin{cases}
-1 & \quad \text{if both } m,n <0 \\
1 & \quad \text{otherwise.}
\end{cases}
\]

\subsection{(Non)periodicity of the Kronecker symbol} \label{Sect:KronSymbol}

Recall that
$
\kappa_a(n) :=
{\leg{a}{n}} .
$
It is common practice to only define this symbol for $a \equiv 0 \mod 4$ or $1 \mod 4$ and squarefree,
as in \cite[Definition~20, p.\ 70]{Landau}, or even just for fundamental discriminants $a$, as in
\cite[p.\ 296]{MV}. Although the difficulty arising for other $a$ is occasionally hinted at, as in
\cite[Exercise~10, p.\ 36]{Cohn}, it seems to be rarely (if ever) treated carefully. For example,
Cohen's book \cite[Theorem 1.4.9]{Cohen} mistakenly asserts that the function $\leg{a}{\cdot}$ is
periodic for any integer $a$. This is false in general (see below), but is often true:

\begin{theorem}\label{thm:1stHalf}
Fix $a \not\equiv 3 \mod 4$. Then the function $\kappa_a(n)$ is periodic.
\end{theorem}

\begin{proof}
Suppose $a$ is even. Then $\kappa_a(2n) = 0$, while $\kappa_a(2n+1)$ is periodic in $n$ (with period $4|a|$,
see \cite[Theorem~3.3.9~(5), p.~76]{Halterkoch}). This handles the cases $a \equiv 0,2 \mod 4$.

In the remaining case $a \equiv 1 \mod 4$, it is shown in \cite[Theorem~99, p.~72]{Landau} that $\kappa_a(n)$
has period $|a|$. (Note that in Theorem~99, $a$ is supposed to be congruent to $0$ or $1$ \mod{4}, as indicated
in \cite[Definition~20, p.\ 70]{Landau}.)
\end{proof}

\noindent
Next we prove a converse of this.

\begin{theorem}\label{3mod4}
Fix $a \equiv 3 \mod 4$. The function $\kappa_a(n)$ is not (ultimately) periodic.
\end{theorem}

\begin{proof}
Fix $a \equiv 3 \mod 4$. By a well-known fact about the Jacobi symbol \cite[Theorem~3.3.9~(5)]{Halterkoch}, we
have $\kappa_a(n+4|a|) = \kappa_a(n)$ for all odd $n$. It follows that the sequence $\big(\kappa_a(2n+1)\big)_{n \geq 0}$
is periodic, and that $2|a|$ is a period. What can we say about the \emph{least} period? We claim it must be even.
Indeed, quadratic reciprocity \eqref{eq:QuadRecip} implies that
\[
\kappa_a\Big(2(n+|a|) + 1\Big)
= - \kappa_a(2n+1)
\]
for any positive $n$, whence
$|a|$ is not a period. We have therefore shown that $(\kappa_a(2n+1))_{n \geq 0}$ is periodic, and that its least
period must be even.

Now suppose that the function $\kappa_a$ were (eventually) periodic. Note that if one has
${\kappa_a(n+2T) = \kappa_a(n)}$ for all large $n$, then
\[
\kappa_a(2) \kappa_a(n) = \kappa_a(2n) = \kappa_a(2n+2T) = \kappa_a(2) \kappa_a(n+T)
\]
whence $\kappa_a(n) = \kappa_a(n+T)$. This shows that the smallest (eventual) period of $\kappa_a(n)$ would have
to be an odd number. Let $q$ denote this smallest (eventual) period. It follows that
\[
\kappa_a\big(2(n+q)+1\big)
= \kappa_a(2n+1+2q)
= \kappa_a(2n+1) ,
\]
so $q$ is an eventual period of the sequence $\big(\kappa_a(2n+1)\big)_{n\geq 0}$.
But then the smallest period of the sequence $\big(\kappa_a(2n+1)\big)_{n\geq 0}$ must divide $q$, and
in particular must be odd! This contradicts what we proved in the first paragraph, and the claim is proved.
\end{proof}

\noindent
Combining Theorems~\ref{thm:1stHalf} and \ref{3mod4} with basic properties of the Kronecker symbol yields
the following result, which is probably known but which we were unable to find in the literature.

\begin{corollary}\label{KroneckerEqualsDirichlet}
The Kronecker symbol $\kappa_a$ is a Dirichlet character if and only if $a \not\equiv 3 \mod 4$.
\end{corollary}

\noindent
In fact, as we shall show below, if $\kappa_a$ is not a Dirichlet character then it must be a mock character.

\begin{remark}

\ {  }

\begin{itemize}

\item The special case of Theorem \ref{3mod4} with $a = 3$ was proved by the second author in an unpublished manuscript (see \cite{Goldmakher}).

\item The fact that the smallest period of the sequence $(\kappa_a(2n+1))_n$ is even for ${a \equiv 3 \mod 4}$
was observed earlier, e.g., in the paper \cite{SW}, where the following period patterns of the sequence
$\big(\kappa_a(2n+1)\big)_{n \geq 0}$ are given:
\[
\begin{array}{lll}
&a = -1 \ \ &\mbox{\rm period pattern\ } + \ - \\
&a = -5 \ \ &\mbox{\rm period pattern\ } + \ + \ 0 \ + \ + \ - \ - \ 0 \ - \ - \\
&a = -9 \ \ &\mbox{\rm period pattern\ } + \ 0 \ + \ - \ 0 \ - \\
&a = +3 \ \ &\mbox{\rm period pattern\ } + \ 0 \ - \ - \ 0 \ + \\
&a = +7 \ \ &\mbox{\rm period pattern\ } + \ + \ - \ 0 \ + \ - \ - \ - \ - \ + \ 0 \ - \ + \ +. \\
\end{array}
\]

\item Some of the sequences $(\kappa_a(n))_n$ appear in the Online Encyclopedia of Integer Sequences
\cite{oeis}, e.g., $\kappa_{-1}(n) =$ A034947$(n)$, $\kappa_3(n) =$ A091338$(n)$, $\kappa_7(n) =$ A089509$(n)$,
and $\kappa_{-5}(n) =$ A226162$(n)$. Moreover, the sequences A117888 and A117889 give the minimal periods of
the sequences $\big(\kappa_a(n)\big)_n$ and $\big(\kappa_{-a}(n)\big)_n$ for small values of $a$, writing $0$
if the sequence is not periodic.

\item If $a \equiv 3 \mod 4$, the sequence $\big(\kappa_a(n)\big)_n$ is a Toeplitz sequence, which (roughly speaking)
is a sequence obtained by repeatedly inserting periodic sequences into periodic sequences. For a more precise definition,
see \cite{All-repet, AB} and the references therein.

\item The proof of Theorem~\ref{3mod4} shows that a nontrivial sequence $(u(n))_{n \geq 0}$ cannot simultaneously
be periodic, completely multiplicative, and satisfy $u(2n) = u(n)$ for all $n$. However, it can satisfy any two of
these properties. In particular, it is possible for a periodic sequence $\big(u(n)\big)_{n \geq 0}$ to satisfy
$u(2n) = u(n)$ for all $n$. The number of such sequences on a given alphabet was studied in \cite{All-repet} in the
context of binary sequences with bounded repetitions and in \cite{EKF} in relation with the so-called perfect shuffle.

\item Periodicity and Kronecker symbols have also been studied in the context of periodic continued fractions; see
    \cite{G1, G2, G3, G4}.

\end{itemize}

\end{remark}

\subsection{The connection to mock characters}
\label{subsect:ConnectToMock}

The (regular) paperfolding sequence $(v_n)_{n \geq 0}$ was introduced in Example~\ref{ex:FirstDefPaper}
and reinterpreted in Example~\ref{ex:PaperFold}. Jonathan Sondow observed that the Kronecker symbol
$\leg{-1}{n}$ satisfies the same recursion as $v_n$, hence generates the same sequence (see \cite[Section~6]{AS}).
As a consequence, we have

\begin{proposition}\label{pap}
Fix an integer $a$. Then either $\kappa_a(n)$ is a Dirichlet character, or a Dirichlet character
multiplied by the paperfolding sequence.
\end{proposition}

\begin{proof}
If $a \not\equiv 3 \mod 4$, then $\kappa_a$ is a Dirichlet character by Theorem~\ref{thm:1stHalf}.
If ${a \equiv 3 \mod 4}$, then (again by Theorem~\ref{thm:1stHalf}) $\kappa_{-a}$ is a Dirichlet character.
To conclude the proof, note that $\kappa_a(n) = \kappa_{-a}(n) \leg{-1}{n}$.
\end{proof}

\noindent
We now arrive at the promised connection between the Kronecker symbol and mock characters.

\begin{theorem}
If $a\equiv 3 \mod 4$, then $\kappa_a$ is a mock character \mock{2}.
\end{theorem}

\begin{proof}
Recall that any periodic sequence is 2-automatic. It follows that $\kappa_a$ is the product of two
$2$-automatic sequences, hence is 2-automatic. The remaining properties of mock characters are
straightforward to verify.
\end{proof}

\subsection{Generating functions involving Kronecker symbols}

We now give some applications of our results to various generating functions (a power series,
a Dirichlet series, an infinite product) involving the Kronecker symbol $\kappa_a(n) = \leg{a}{n}$.

\begin{proposition}
Fix $a \equiv 3 \mod 4$. Let $f$ be any injective map from $\{0, \pm 1\}$ to $\F_4$, the field with four elements.
Then the formal power series $\sum_{n \geq 0} f(\leg{a}{n}) X^n$ has degree $2$ or $4$ over $\F_4(X)$, the field
of rational functions on $\F_4$.
\end{proposition}

\begin{proof}
Let
\[
G(X) :=
\sum_{n \geq 0} f\big(\kappa_a(n)\big) X^n ,
\]
where $\kappa_a(n) = \leg{a}{n}$ as before.
Christol's theorem (Theorem~\ref{christol}) implies that $G$ must be algebraic over $\F_4(X)$.
On the other hand, by Theorem~\ref{3mod4} we know that the coefficients of $G$ are not (eventually) periodic,
whence $G$ is not a rational function. Thus, the minimal polynomial of $G$ has degree at least 2. We now show
its degree is at most 4.

Recall that $\alpha^4 = \alpha$ for all $\alpha \in \F_4$, and that $\F_4$ has characteristic 2. It follows that
\[
G(X)^4
= \sum_{n \geq 0} f\big(\kappa_a(n)\big)^4 X^{4n}
= \sum_{n \geq 0} f\big(\kappa_a(n)\big) X^{4n}
= \sum_{n \geq 0} f\big(\kappa_a(4n)\big) X^{4n} ,
\]
where the last equality holds because
$\kappa_a(4n) = \kappa_a(2)^2 \kappa_a(n) = \kappa_a(n)$.
We deduce that
\begin{equation}
\label{eq:FnlEqn}
G^4 + G + R = 0 ,
\end{equation}
where
$\displaystyle
R(X) :=
\sum_{\substack{n \geq 0 \\ 4 \dnd n}}
    f\big(\kappa_a(n)\big) X^{n} .
$
Relation \eqref{eq:FnlEqn} is almost enough to show that the minimal polynomial of $G$ has degree at most 4;
all that is left to check is that $R(X)$ is a rational function. From the beginning of the proof of
Theorem~\ref{3mod4} we know that the sequence ${\big(\kappa_a(2n+1)\big)_{n \geq 0}}$ is periodic.
Thus the sequence ${\big(\kappa_a(4n+2)\big)_{n \geq 0} = \kappa_a(2) \cdot \big(\kappa_a(2n+1)\big)_{n \geq 0}}$
is also periodic. We conclude that the coefficients of $R(X)$ are periodic, whence $R(X)$ is rational.

We have thus bounded the degree of the minimal polynomial of $G$ between 2 and 4. To conclude the proof, we
show that the degree cannot equal 3. First, recall from~\eqref{eq:FnlEqn} that $Y=G$ is a solution to the equation
\[
Y^4 + Y + R = 0 .
\]
Now observe that no rational function $Y(X) \in \F_4(X)$ satisfies this equation. Indeed, if $Y$ is a solution,
then so is $Y+\lambda$ for any $\lambda \in \F_4$; it follows that either all roots of $Y^4 + Y + R$
are rational, or none of them are. Since $G$ is an irrational root, we are in the latter case, so $Y^4 + Y + R$
has no linear factors with rational coefficients. We deduce that it also has no cubic factors with rational
coefficients, and the claim follows.
\end{proof}

\begin{proposition}\label{prop:MockLFn}
The series
$\sum_{n \geq 1} \frac{\leg{a}{n}}{n^s}$
is either a Dirichlet $L$-function or the product of
$\frac{2^s}{2^s - (-1)^{\frac{a^2 - 1}{8}}}$
with a Dirichlet $L$-function.
\end{proposition}

\begin{proof}
As above, set $\kappa_a(n) := \leg{a}{n}$, and define
\[
L_a(s) := \sum_{n \geq 1} \frac{\kappa_a(n)}{n^s}
\]
If $a \not\equiv 3 \mod 4$, Proposition~\ref{pap} implies that $L_a(s)$ is a Dirichlet $L$-function.

Now suppose instead that $a \equiv 3 \mod 4$.
Define a function $\chi : \Z \to \C$ by
\[
\chi(n) :=
\begin{cases}
\kappa_a(n) & \quad \text{if $n$ is odd} \\
0 & \quad \text{if $n$ is even.}
\end{cases}
\]
It is easy to verify that $\chi$ is completely multiplicative, and (from the beginning of the proof of
Theorem~\ref{3mod4}) it is also periodic. Proposition~\ref{prop:CharFirstTwoProp} implies that $\chi$
is a Dirichlet character, and we write $L(s,\chi)$ to denote the associated $L$-function.
We have:
\[
\begin{split}
L_a(s)
&=
\sum_{n \geq 1} \frac{\kappa_a(n)}{n^s}
=
\sum_{n \geq 1} \frac{\kappa_a(2n)}{(2n)^s}
+ \sum_{n \geq 0} \frac{\kappa_a(2n+1)}{(2n+1)^s} \\
&=
\frac{\kappa_a(2)}{2^s}
\sum_{n \geq 1} \frac{\kappa_a(n)}{n^s}
+ \sum_{m \geq 1} \frac{\chi(m)}{m^s} \\
&=
\frac{\kappa_a(2)}{2^s}
L_a(s)
+ L(s,\chi) .
\end{split}
\]
Thus
$
L_a(s) =
\left(1 - \frac{\kappa_a(2)}{2^s}\right)^{-1} L(s,\chi) ,
$
and the claim follows.
\end{proof}

Recall from Example~\ref{ex:PaperFold} the paperfolding sequence $(v_n)_{n \geq 0}$, which (as noted at the
start of Section~\ref{subsect:ConnectToMock}) is the same as the sequence $\kappa_{-1}(n)$. It turns out that
many well-known properties of the paperfolding sequence hold more generally for $\big(\kappa_a(n)\big)_{n \geq 0}$
when $a \equiv 3 \mod 4$. For example, the following identity involving the paperfolding sequence was proved
in \cite{prod}:
\begin{equation}\label{PP}
\prod_{n \geq 1} \left(\frac{2n}{2n+1}\right)^{v_{n+1}} =
\frac{\Gamma(1/4)^2}{8\sqrt{2 \pi}}
\end{equation}
(the terms in the product are fractions, not Kronecker symbols).
A similar proof gives:

\begin{proposition}
Given $a \equiv 3 \mod 4$, set $\alpha := (-1)^{\frac{a^2-1}{8}}$. Then
\[
\prod_{n \geq 1}
\left(\left(\frac{n}{n+1}\right)\left(\frac{2n+2}{2n+1}\right)^{\alpha}\right)^{\leg{a}{n+1}}
= \frac{1}{2}\prod_{n \geq 1} \left(\frac{2n}{2n+1}\right)^{\alpha\leg{a}{2n+1}}.
\]
Furthermore the left side is a finite product of terms of the form $\Gamma(x/4a)^{\pm 1}$, where $x\in\Z$.
\end{proposition}

\begin{proof}
The proof mimics the proof of the result for $a = -1$ in \cite{prod}. It uses the fact (from the proof of
Proposition~\ref{prop:MockLFn}) that
\[
\chi(n) :=
\begin{cases}
\leg{a}{n} & \quad \text{if $n$ is odd} \\
0 & \quad \text{if $n$ is even.}
\end{cases}
\]
is a Dirichlet character, from which it follows that the sum of the values of $\leg{a}{2n+1}$ on a period is zero.
\end{proof}

\section{Quantified Mockery} \label{Sect:FinalSection}

Proposition~\ref{pap} shows that, in a qualitative sense, the mock character $\kappa_a$ is essentially a Dirichlet
character. This statement can be quantified using the language of the theory of pretentiousness. Recall the
pseudometric introduced by Granville and Soundararajan \cite{GranSound}: given any two completely multiplicative
functions $f,g : \Z \to \U$ (with $\U$ denoting the complex unit disc), set
\[
\D(f,g; y) :=
\left(
\sum_{p \leq y} \frac{1- \Re f(p) \overline{g(p)}}{p}
\right)^{1/2} .
\]
This is a useful tool for quantifying how closely one function mimics another. Any time we have
\[
\D(f,g;y)^2 = o(\log \log y) ,
\]
this means that $f$ and $g$ behave similarly, and the smaller the `distance' between $f$ and $g$, the more similar
their behavior. One key property of this pseudometric is a `triangle inequality' \cite{GranSound}:
for any functions $f_i,g_i : \Z \to \U$,
\[
\D(f_1,f_2;y) + \D(g_1,g_2;y) \geq \D(f_1 f_2, g_1 g_2; y) .
\]
We have
\begin{proposition}
Let $\kappa_a(n) = \leg{a}{n}$ as above.
Then for every $a \in \Z$ there exists a Dirichlet character $\chi$ such that
\[
\D(\kappa_a, \chi; y) \ll 1 .
\]
\end{proposition}

\begin{proof}
If $a \not\equiv 3 \mod 4$ then Theorem~\ref{thm:1stHalf} implies that $\kappa_a$ is a Dirichlet character, so
we can take $\chi = \kappa_a$ to conclude. Next, observe that $\kappa_{-1}(p) = \chi_{-4}(p)$ for all odd primes,
where $\chi_{-4}$ is the nontrivial character \mod 4. It follows that
\[
\phantom{\qquad
\forall y \geq 2}
\D(\kappa_{-1} , \chi_{-4} ; y)^2
= \sum_{p \leq y} \frac{1 - \kappa_{-1}(p) \chi_{-4}(p)}{p}
= \frac{1}{2}
\qquad
\forall y \geq 2 ,
\]
so the claim holds for $\kappa_{-1}$. Finally, suppose $a \equiv 3 \mod 4$. Then (again by Theorem~\ref{thm:1stHalf})
$\kappa_{-a}$ is a Dirichlet character, whence $\chi := \chi_{-4} \kappa_{-a}$ is a Dirichlet character. We deduce
\[
\begin{split}
\D(\kappa_a , \chi ; y)
&= \D(\kappa_{-1} \kappa_{-a} , \chi_{-4} \kappa_{-a} ; y) \\
&\leq \D(\kappa_{-1} , \chi_{-4}; y) +
        \D(\kappa_{-a},\kappa_{-a}; y) \\
&= \D(\kappa_{-1} , \chi_{-4}; y) +
        \sum_{\substack{p \leq y \\ \kappa_{-a}(p) = 0}} \frac{1}{p}
\end{split}
\]
The first sum is bounded by our work above; the second is bounded because $\kappa_{-a}(p) = 0$ only for those $p$
dividing the conductor of $\kappa_{-a}$.
\end{proof}

\noindent
We suspect this is a special case of a more general result:

\begin{conj}
For any mock character $\kappa$, there exists a Dirichlet character $\chi$ such that
$\D(\kappa,\chi;y)$ is bounded. Conversely, if $\kappa : \Z \to \U$ is completely multiplicative
and a bounded distance from some Dirichlet character, then $\kappa$ must be a mock character.
\end{conj}

\bigskip

\noindent
{\bf Acknowldgments} \ The first named author warmly thanks G\'erald Tenenbaum for ``old''
discussions on automatic multiplicative functions and Jonathan Sondow for having noted that $\leg{-1}{n}$
was the regular paperfolding sequence. We heartily thank Henri Cohen, Olivier Ramar\'e, Jeff Shallit, and
Soroosh Yazdani, for discussions or for comments on previous versions of this paper.


\begin{thebibliography}{99}

\bibitem{All-repet} J.-P. Allouche, Suites infinies \`a r\'ep\'etitions born\'ees, {\it S\'em. Th\'eor.
Nombres Bordeaux} (1983--1984) Exp. No. 20, 20-01--20-11.

\bibitem{prod} J.-P. Allouche, Paperfolding infinite products and the gamma function, {\it J. Numb. Theory}
{\bf 148} (2015) 95--111.

\bibitem{AB} J.-P. Allouche, R. Bacher, Toeplitz sequences, paperfolding, Hanoi towers and
progression-free sequences of integers, {\it Ens. Math.} {\bf 38} (1992), 315--327.

\bibitem{AMF} J.-P. Allouche, M. Mend\`es France, Automata and automatic sequences, in {\it Beyond
Quasicrystals (Les Houches, 1994)}, Springer, Berlin, 1995, pp. 293--367.

\bibitem{AS} J.-P. Allouche, J. Shallit, {\it Automatic Sequences. Theory, Applications, Generalizations},
Cambridge University Press, Cambridge, 2003.

\bibitem{BBC} J. P. Bell, N. Bruin, M. Coons, Transcendence of generating functions whose coefficients
are multiplicative, {\it Trans. Amer. Math. Soc.} {\bf 364} (2012) 933--959.

\bibitem{BCH} J. P. Bell, M. Coons, K. G. Hare, The minimal growth of a $k$-regular sequence,
{\it Bull. Aust. Math. Soc.} {\bf 90} (2014) 195--203.

\bibitem{BC} P. Borwein, M. Coons, Transcendence of power series for some number theoretic functions,
{\it Proc. Amer. Math. Soc.} {\bf 137} (2009) 1303--1305.

\bibitem{BCC} P. Borwein, S. K. K. Choi, M. Coons, Completely multiplicative functions taking values
in $\{-1, 1\}$, {\it Trans. Amer. Math. Soc.} {\bf 362} (2010) 6279--6291.

\bibitem{Bronstein} B. S. Bron\v{s}te\u{\i}n,
Unboundedness of the sum function of a generalized character (in Russian),
{\it Moskov. Gos. Univ. U\v{c}. Zap. Mat.} {\bf 165} (1954) 212--220

\bibitem{C1} G. Christol, Ensembles presque p\'eriodiques $k$-reconnaissables,
{\it Theoret.\ Comput.\ Sci.} {\bf 9} (1979) 141--145.

\bibitem{C2} G. Christol, T. Kamae, M. Mend\`es France, G. Rauzy, Suites alg\'e\-bri\-ques, automates
et substitutions, {\it Bull. Soc.  Math. France} {\bf 108} (1980) 401--419.

\bibitem{Cobham} A. Cobham, On the base-dependence of sets of numbers recognizable by finite automata,
{\it Math. Systems Theory\,} {\bf 3} (1969) 186--192.

\bibitem{Cohen} H. Cohen, {\it A Course in Computational Algebraic Number Theory}, 3rd ed., Graduate Texts
in Mathematics vol. 138, Springer-Verlag, Berlin, 1996.

\bibitem{Cohn} H. Cohn, {\it Advanced Number Theory}, Reprint of {\it A Second Course in Number Theory}, 1962,
Dover Publications, Inc., New York, 1980.

\bibitem{Conway} J. H. Conway, {\it The Sensual (Quadratic) Form}, Carus Mathematical Monographs
(vol 26), Mathematical Association of America, 1997.

\bibitem{Coons} M. Coons, (Non)automaticity of number theoretic functions,
{\it J. Th\'eor. Nombres Bordeaux\,} {\bf 22} (2010) 339--352.

\bibitem{Cudakov} N. G. \v{C}udakov, On a class of completely multiplicative functions (in Russian),
{\it Uspehi Matem. Nauk (N.S.)} {\bf 8} (1953) 149--150.

\bibitem{CL} N. G. \v{C}udakov, Yu. V. Linnik,
On a class of completely multiplicative functions (in Russian),
{\it Doklady Akad. Nauk SSSR (N.S.)} {\bf 74} (1950) 193--196.

\bibitem{CR} N. G. \v{C}udakov, K. A. Rodosski\u{\i},
On generalized characters (in Russian),
{\it Doklady Akad. Nauk SSSR (N.S.)} {\bf 73} (1950) 1137--1139

\bibitem{Delange} H. Delange, Sur les fonctions arithm\'etiques multiplicatives,
{\it Ann. Sci. \'Ecole Norm. Sup.} {\bf 78} (1961) 273--304.

\bibitem{EKF} J. Ellis, T. Krahn, H. Fan, Computing the cycles in the perfect shuffle permutation,
{\it Inform. Process. Lett.} {\bf 75} (2000) 217--224.

\bibitem{G1} K. Girstmair, Continued fractions and Jacobi symbols, {\it Int. J. Number Theory}
{\bf 7} (2011) 1543--1555.

\bibitem{G2} K. Girstmair, Periodic continued fractions and Jacobi symbols, {\it Int. J. Number Theory}
{\bf 8} (2012) 1519--1525.

\bibitem{G3} K. Girstmair, Jacobi symbols and Euler's number $e$, {\it J. Number Theory} {\bf 135}
(2014) 155--166.

\bibitem{G4} K. Girstmair, Periodic continued fractions and Kronecker symbols, arxiv:1504.02618, 2014.

\bibitem{GoldThesis} L. Goldmakher, {\it Multiplicative Mimicry and Improvements of the P\'olya-Vinogradov
Inequality}, Ph.D. Thesis, University of Michigan (2009).

\bibitem{Goldmakher} L. Goldmakher, Legendre, Jacobi, and Kronecker symbols, electronically available
at the URL {\tt http://web.williams.edu/Mathematics/lg5/Kronecker.pdf}

\bibitem{GranSound} A. Granville, K. Soundararajan,
Large character sums: pretentious characters and the P\'{o}lya-Vinogradov theorem,
\emph{J. Amer. Math. Soc.} \textbf{20} (2007), no. 2, 357--384.

\bibitem{Halterkoch} F. Halter-Koch, {\it Quadratic Irrationals. An Introduction to Classical Number Theory},
CRC Press, Taylor \& Francis, 2013.

\bibitem{HeppnerMaxsein} E. Heppner, T. Maxsein,
Potenzreihen mit multiplikativen Koeffizienten,
\emph{Analysis} \textbf{5} (1985), 87--95.

\bibitem{Landau} E. Landau, {\it Elementary Number Theory}, translated by J. E. Goodman, with Exercices
by P.~T.~Bateman and E.~E.~Kohlbecker, Chelsea Publishing Company, 1958.

\bibitem{Methfessel} C. Methfessel,
Multiplicative and additive recurrent sequences,
\emph{Arch. Math.} \textbf{63} (1994), 321--328.

\bibitem{MV} H. L. Montgomery, R. C. Vaughan, {\it Multiplicative Number Theory I. Classical Theory},
Cambridge University Press, 2006.

\bibitem{oeis} {\it On-Line Encyclopedia of Integer Sequences}, electronically available at the URL
{\tt http://oeis.org}

\bibitem{Sarkozy} A. S\'{a}rk\"{o}zy, On multiplicative arithmetic functions satisfying a linear recursion,
{\it Studia Sci. Math. Hungar.} {\bf 13} (1978), 79--104.

\bibitem{SP1} J.-C. Schlage-Puchta, A criterion for non-automaticity of sequences,
{\it J. Integer Seq.} {\bf 6} (2003), Article 03.3.8, 5 pp.

\bibitem{SP2} J.-C. Schlage-Puchta, Completely multiplicative automatic functions,
{\it Integers} {\bf 11} (2011), A31, 8.

\bibitem{SW} D. Shanks, J. W. Wrench Jr., The calculation of certain Dirichlet series,
{\it Math. Comp.} {\bf 17} (1963) 136--154.

\bibitem{Vandehey} J. Vandehey, On multiplicative functions with bounded partial sums,
{\it Integers\,} {\bf 12} (2012) 741--755.

\bibitem{Yazdani} S. Yazdani, Multiplicative functions and $k$-automatic sequences,
{\it J. Th\'eor. Nombres Bordeaux\,} {\bf 13} (2001) 651--658.



\end{thebibliography}
\end{document}